\documentclass[11pt, reqno]{amsart}

\usepackage[T1]{fontenc}
\usepackage{amsthm, amsfonts, mathrsfs}
\usepackage{amssymb, amsmath, theoremref}
\usepackage{times, txfonts, pifont, graphicx}

\newcommand{\R}{{\mathbf{R}}}
\newcommand{\T}{{\mathbf{T}}}

\newcommand{\SO}{\mathop{\! \, \rm SO }\nolimits}

\newcommand{\SE}{\mathop{\! \, \rm SE }\nolimits}
\newcommand{\se}{\mathop{\! \, \rm se }\nolimits}
\newcommand{\tr}{\mathop{\! \, \rm tr }\nolimits}
\newcommand{\sn}{\mathop{\! \, \rm sn }\nolimits}
\newcommand{\cn}{\mathop{\! \, \rm cn }\nolimits}
\newcommand{\dn}{\mathop{\! \, \rm dn }\nolimits}

\newcommand{\lefthook}{\mbox{$\, \rule{8pt}{.5pt}\rule{.5pt}{6pt}\, \, $}}

\theoremstyle{plain}
\newtheorem{claim}{\sc Claim}[section]
\newtheorem{corollary}[claim]{\sc Corollary}

\newtheorem{proposition}[claim]{\sc Proposition}
\newtheorem{theorem}[claim]{\sc Theorem}

\theoremstyle{definition}
\newtheorem{definition}[claim]{\sc Definition}

\newtheorem{example}[claim]{\sc Example}
\theoremstyle{remark}

\begin{document}

\title[Not-quite-Hamiltonian]{Not-quite-Hamiltonian reduction}
\author[L. Bates and J. \'Sniatycki]{Larry M. Bates and J\k{e}drzej \'Sniatycki}

\date{\today }

\begin{abstract}
The not-quite-Hamiltonian theory of singular reduction and
reconstruction is described.  This includes the notions of both
regular and collective Hamiltonian reduction and reconstruction.
\end{abstract}

\maketitle

\section{Introduction}
\noindent Ever since the beginning of analytical mechanics, there has been an 
effort to understand how to reduce the equations of motion given the presence 
of constraints of one form or another.  For example, the motion of a particle 
that is described by Lagrangian or Hamiltonian formalism that is constrained to 
move on a submanifold of configuration space will, by employing 
D'Alembert's principle, give a reduced system of equations of the same form.  A 
much more sophisticated, and perhaps the most striking early example of such 
considerations was Jacobi's elimination of the node in the three body problem.  
Such examples showed that there was a relation between symmetry and conservation 
laws, and these were explained for variational problems by Noether in her 
important work \cite{noether}.  Somewhat dual to this, because of the 
Hamilton-Jacobi theory, there was an evolving understanding of the nature of 
symmetry and conservation laws on the Hamiltonian side, especially in 
understanding non-abelian symmetry groups and the reduction of Hamilton's 
equations.  The first serious counterpart to Noether's theorem on the 
Hamiltonian side was the paper of Meyer \cite{meyer}.\footnote{Such a 
judgement call is always to some extent a question of taste.  The reader may 
have some sympathy for our point of view after rereading the earlier work of 
Arnol'd \cite{arnold66} and Smale \cite{smale}.} In this 
paper, Meyer showed that the free and proper Hamiltonian action of a connected 
Lie group on a symplectic manifold led to a reduced symplectic 
manifold and the reduced dynamics was Hamiltonian.  The importance of this 
theorem is the realization that the structure of the equations of motion has 
been reproduced under reduction by symmetry.  This is a recurrent theme in 
further work. This work was followed by a torrent of papers on 
reduction, all with somewhat different emphases.  For example, some, such as 
Marsden and Weinstein \cite{marsden-weinstein}, stressed the role of the 
momentum map, while others, such as Churchill, Kummer and Rod 
\cite{churchill-kummer-rod}, looked at the relations of symmetry to averaging.  
During the 1980s and 1990s there was a growing awareness of the need to include 
singularity and the desire to discuss dynamics on the reduced space.  A key 
observation in this time was that the dynamics on the reduced space could be 
described by the Poisson bracket on the invariant functions.  Some of the 
notable works using this idea were those of Gotay and Bos 
\cite{gotay-bos}, 
Arms, Cushman and Gotay \cite{arms-cushman-gotay}, and Sjamaar and Lerman 
\cite{sjamaar-lerman}.  At this point it had become clear that the reduced space 
had dynamics, and that it could be described stratum by stratum using the 
Poisson bracket.\footnote{It is our contention that this is about as far as the 
theory can be developed without the notion of differential spaces.  The results 
of this stage of the development of singular reduction are completely described 
in the monograph of Ratiu and Ortega \cite{ratiu-ortega}.}  

Since then, it is now known that the reduced space 
is not only a topological space, but also has a differential structure, which 
is 
completely described by an algebra of smooth functions.  These smooth functions 
are push-forwards of functions on the original space that are invariant 
under the 
group action.    Such singular spaces are described naturally by the theory of 
subcartesian differential spaces, and in the case under consideration, {\it 
 support dynamics as well\,} because the algebra of smooth functions has a 
Poisson structure.  It is our view that satisfying the dual requirements of 
describing  the analytic structure of the singularities of the reduced space 
{\it and}  defining the reduced dynamics provides a powerful justification for 
our use of differential spaces.  

A related development in the theory of constrained Hamiltonian systems with 
symmetry 
 has been the reduction of non-holonomic constraints. The 
regular theory for transverse linear constraints was considered by Koiller 
\cite{koiller}, and extended to the nontransverse case by Bates and \'Sniatycki 
\cite{bates-sniatycki93}.  Regular reduction of nonlinear non-holonomic 
constraints was given by de Leon and de Diego in \cite{deleon97}, and singular 
examples involving linear constraints were considered by Bates in 
\cite{bates98}.  
The singular reduction of nonlinear nonholonomic constraints was 
given by Bates and Nester in \cite{bates-nester}.  What is notable here is 
that the 
formulation is once again in terms of invariant functions and the Poisson 
bracket, the wrinkle being that the Hamiltonian operator need no longer be an 
invariant function, and so the reduced dynamics is given by an {\it outer} 
Poisson morphism.   

Of course, constraints in mechanics do not have to have anything to do with 
symmetry.  There is a less mature, but somewhat parallel stream of development 
that tries to understand the nature of the constraints that show up in systems 
where the Lagrangian is degenerate in the sense that the Legendre 
transformation does not define a local diffeomorphism.  This theory, 
inaugurated by Dirac in \cite{dirac50}, (giving what is now 
called the Dirac constraint algorithm), describes a way to produce a 
Hamiltonian on a submanifold of the phase space.  The constraint algorithm has 
been geometrized by Gotay, Nester and Hinds \cite{gotay-nester-hinds78} and 
Lusanna \cite{lusanna}.  However, the nature of such constraints in the 
Lagrangian is such that the initial data set, which is the subset of the 
original space on which the Lagrangian is defined actually has local solutions 
of the Euler-Lagrange equations, can be a singular space.  Our experience is 
that the best way to deal with such constraints and their singularities, as well 
as the related constructions of reduced spaces, first class functions, etc., is 
to employ the theory of differential spaces \cite{bates-sniatycki2013}.

This note generalizes the singular reduction and reconstruction of a 
Hamiltonian 
 dynamical system to the case in which the Hamiltonian is not 
necessarily invariant under 
the proper Hamiltonian action of a connected Lie group on a symplectic 
manifold, but 
nevertheless still manages to have reduced dynamics.  Consistent 
with the previous cases, the singular reduced dynamics is given 
in terms of a Poisson bracket on the invariant functions. The 
main difference in the not-quite Hamiltonian case with the singular Hamiltonian 
case is that the reduced dynamics is not given by the Poisson bracket of an 
invariant function with an invariant Hamiltonian, as now the bracket of the 
Hamiltonian with an invariant function is an outer Poisson morphism on the 
invariant functions.  Furthermore, in a manner 
similar to Hamiltonian reconstruction, integration of the Hamiltonian dynamics 
is given by integration of an equation on the dual of the Lie algebra, after 
which the original dynamics is reconstructed from the reduced dynamics via 
integration with respect to a moving isotropy subgroup of the original group.

\section{Preliminaries}
Denote a Hamiltonian system by $(P,\omega, h)$.  Here $P$ is the phase space, 
$\omega$ the symplectic form, and $h$ is the Hamiltonian. The Hamiltonian 
vector field $X_h$ satisfies Hamilton's equations $X_h\lefthook\omega=dh$.
Denote by $G$ a connected Lie group, and by $\phi$ its action on $P$.  A 
blanket assumption in this paper is that the action $\phi$ is proper and 
Hamiltonian.  Denote the momentum map for the action $\phi$ by $j:P\rightarrow 
\mathfrak{g}^*$.

The quotient space $\bar{P}:=P/G$, the space of $G$-orbits, is given 
the quotient topology.  Because the action of the group $G$ 
is proper, $\bar{P}$ has a much richer structure than merely that of a 
topological space.  In fact, $\bar{P}$ is known to be a {\it 
stratified subcartesion differential space} (see 
\cite{cushman-bates2015} or \cite{sniatycki13}).  In 
particular, this means that the ring of continuous functions on 
$\bar{P}$, 
denoted $C^{\infty}(\bar{P})$ (declared to be the smooth functions), which are 
push-forwards of smooth $G$-invariant functions on $P$, satisfy the 
conditions  
\begin{enumerate}
  \item The family 
  \[
    \{f^{-1}(I)\,|\, f\in C^{\infty}(\bar{P})\text{ and }I\text{ is an 
open interval in }\R \}
  \]
  is a subbasis for the topology of $\bar{P}$.
  \item If $f_1,\dots,f_n\in C^{\infty}(\bar{P})$ and $F\in C^{\infty}(\R^n)$, 
then $F(f_1,\dots,f_n)\in C^{\infty}(\bar{P})$.
  \item If $f:\bar{P}\rightarrow \R$ is a function such that for each 
$p\in\bar{P}$, there is an open neighbourhood $U$ of $p$ and a function 
$f_p\in C^{\infty}(\bar{P}) $ satisfying $f_p|_U=f|_U$, then $f\in 
C^{\infty}(\bar{P})$.
\end{enumerate}
$\bar{P}$ is {\it subcartesian} means that it is Hausdorff and 
each point $p\in \bar{P}$ has a 
neighbourhood $U$ diffeomorphic to a subset $V$ of $\R^n$.  The stratification 
of $\bar{P}$ is given by orbit type.  Since the many technical details in 
the proof of this would lead us too far astray, we refer the reader to the 
discussions in \cite{bierstone}, \cite{cushman-bates2015} or \cite{sniatycki13}. 
The reader should also 
note that because the group action is Hamiltonian, the stratification of the 
quotient space $\bar{P}$ is determined by the Poisson bracket on the invariant 
functions.  However, we state below definitions and results that are essential 
for this paper.
\begin{definition}
  A differential space $M$ is a topological space endowed with the ring 
$C^{\infty}(M)$ of continuous functions that satisfy the three conditions above.
\end{definition}

\begin{definition}
  A map $\psi: M\rightarrow N$ between differential spaces $M$ and $N$ is 
{\it smooth} if $\psi^*f\in C^{\infty}(M)$ for $f\in C^{\infty}(N)$. A smooth 
map between differential spaces is a {\it local diffeomorphism} if it is a 
local homeomorphism with a smooth inverse.
\end{definition}

\begin{theorem}  For every derivation $X$  of the ring of smooth 
functions on a subcartesian differential space $M$, and each point $p \in M$, 
there exists a unique maximal integral curve of $X$ through $p$.
\end{theorem}

\begin{proof}
The proof may be found in \cite{sniatycki13}.
\end{proof}

For $p \in M$ and $t$ in the domain of the unique integral curve of $X$ through 
$p$, denote by $(\exp tX)p$ the point on the  integral curve of $X$ through 
$p$ corresponding to $t$. This gives a local one-parameter group $\exp tX$ of 
local transformations of $M$.
\begin{definition} A derivation $X$ of $C^{\infty}(M)$ is a {\it vector field} 
on $M$ if $\exp tX$ is a local one-parameter group of local diffeomorphisms of 
$M$.
\end{definition}
\begin{theorem} 
  Orbits of a family of vector fields on $M$ are 
smooth manifolds immersed in $M$.
\end{theorem}
\begin{proof}
The proof may be found in \cite{sniatycki13}.
\end{proof}
\begin{theorem}
   If $M =\bar{P}$ is the space of orbits of a proper action of a connected 
Lie group $G$ on a manifold $P$, then orbits of the family of all vector fields 
on $M$ coincide with the strata of the orbit type stratification of $P$.
\end{theorem}
\begin{proof}
The proof may be found in \cite{sniatycki13}.
\end{proof}

A main concern of this paper is when the connected Lie group $G$ has a proper 
Hamiltonian action $\phi$ on $P$ and this action can be divided out to produce a 
reduced space $\bar{P}$ that also has reduced dynamics. As a first step we 
extend a well-known theorem for free and proper actions to the 
case of merely proper actions.  Let 
\[
\phi :G\times P\rightarrow P:(g,p)\mapsto \phi(g,p)=:\phi _{g}(p)
\]
be a proper action of the connected Lie group $G$ on the manifold $P$ and let 
$\rho :P\rightarrow 
\bar{P}$ 
be the orbit map. Then  $\rho ^{\ast }(C^{\infty }(\bar{P}))=C^{\infty
}(P)^{G}$. For a vector field $X$ on $P,$
\[
\phi _{g\ast }X(p)=T\phi _{g}(X(\phi _{g^{-1}}(p))),
\]
and for a function $f\in C^{\infty }(P),$
\[
(\phi _{g\ast }X)\cdot f=\phi _{g^{-1}}^{\ast }(X\cdot \phi _{g}^{\ast }f).
\]

\begin{proposition}\label{appendix-proposition}
If $X$ is a vector field on $P$ such that $\phi _{g\ast }X-X$ is tangent to
orbits of the action of $G$, then $X$ descends to a vector field $\bar{X}
=\rho _{\ast }X$ on $\bar{P}.$
\end{proposition}

\begin{proof}
For a $G$-invariant function $f$ on $P$ and $g\in G,$
\[
\begin{split}
\phi _{g^{-1}}^{\ast }(X\cdot f) &=(\phi_{(g^{-1})\ast}X)\cdot f \\
& =(\phi_{(g^{-1})\ast }X)\cdot \phi _{g}^{\ast }f  \\
& =(\phi _{g\ast }X)\cdot f \\
& =(\phi _{g\ast}X-X)\cdot f+X\cdot f \\
& =X\cdot f,
\end{split}
\]
because $(\phi _{g\ast }X-X)$ is tangent to orbits of the action of $G$ and $f$ 
is $G$-invariant. Hence, $X\cdot f$ is $G$-invariant. Thus $X$ is a
derivation of $C^{\infty }(P)^{G}$, which implies that it descends to a
derivation $\bar{X}=\rho _{\ast }X$ of $C^{\infty }(\bar{P}).$ Integration
of the derivation $\bar{X}$ gives rise to a maximal integral curve $\bar{c}$ of 
$\bar{X}$ through $\bar{p}$ such
that $\bar{c}(t)=\rho \circ c(t)$, where $c$ is the maximal integral curve
of $X$ through $p$.

It remains to prove that translations along integral curves of $\bar{X}$
gives rise to a local one-parameter group $\exp t\bar{X}$ of local
diffeomorphisms of $\bar{P}.$ The vector field $X$ on $P$ generates a local 
one-parameter group $\exp tX$ of local diffeomorphisms of $P$ such that 
$t\mapsto 
(\exp tX)(p)$
is the maximal integral curve of $X$ through $p$. Since the derivation $X$
of $C^{\infty }(P)$ preserves $C^{\infty }(P)^{G}$, and $G$-invariant
functions on $P$ separate $G$-orbits, it follows that $\exp tX$ maps $G$-orbits 
to $G$-orbits. 

It is important to note that the assertion of mapping $G$-orbits to $G$-orbits 
does not require the flow of $X$ to be complete, because $\exp tX$ is 
interpreted in the sense that if two points $p$ and $q$ are in the same 
$G$-orbit and $(\exp tX)(p)$ and $(\exp tX)(q)$ are both defined, then $(\exp 
tX)(p)$ and $(\exp tX)(q)$ are in the same $G$-orbit. It is in this way that 
the 
reduced local flow $\exp t\bar{X}$ is defined.  Note that 
the reduced local flow may be defined for all time even though the original 
vector field may be incomplete everywhere, and have no positive minimum 
time of existence on any $G$-orbit.\footnote{On $\{(x,y)\,| \, y>0\}$ consider 
the vector field $X=\partial_x+y^2\partial_y$, with group action generated by 
$y\partial_y$.}

Therefore, $(\exp t\bar{X})(\bar{p})=(\exp t\bar{X}
)(\rho (p))=\rho \circ (\exp tX)(p)$. In other
words, 
\[
\rho \circ \exp tX=(\exp t\bar{X})\circ \rho .
\]
Since $\exp tX$ is a local one-parameter group of local transformations of $P$, 
it follows that $\exp t\bar{X}$ is a local one-parameter group of local
transformations of $\bar{P}$. Since $\bar{f}\in C^{\infty }(\bar{P})$ implies 
that the pull-back $f=\rho ^{\ast }\bar{f}\in C^{\infty }(P)^{G}$, it follows 
that  
\[
\rho ^{\ast }(\exp t\bar{X})^{\ast }\bar{f}=(\exp tX)^{\ast }\rho ^{\ast }
\bar{f}=(\exp tX)^{\ast }f
\]
is $G$-invariant. This implies that $(\exp t\bar{X})^{\ast }\bar{f}\in 
C^{\infty }(\bar{P})$, and ensures that $\exp t\bar{X}$ is a local 
diffeomorphism of $\bar{P}.$
\end{proof}
Since the projection of the orbit type stratification of $P$ to $\bar{P}$ is a 
stratification of $\bar{P}$, and these strata coincide with the orbits of 
the family of all vector fields on $\bar{P}$, it follows that
\begin{corollary}
  The reduced vector field $\bar{X}$ preserves the stratification of $\bar{P}$ 
by orbit type.
\end{corollary}

\section{Not-quite-Hamiltonian reduction}
Suppose that the Hamiltonian system $(P,\omega,h)$ has a proper Hamiltonian 
action $\phi$ of the connected Lie group $G$ with momentum map 
$j=(j_1,\dots,j_n)$. Let 
the infinitesimal generators of the $G$-action be $X_a$, and 
$X_a\lefthook\omega=dj_a$. Suppose 
that the Hamiltonian $h$ is not $G$-invariant, but the Poisson brackets satisfy
\[
  \{j_a,h\} = f_a(j_1,\dots,j_n)
\]
for some functions $f_a$ of the momenta.\footnote{This global formulation has 
been given for the sake of cleanliness.  Locally, it seems that it is a 
question  of to what extent the condition ${\phi_g}_*X_h-X_h\equiv 0 
\,\,(\text{mod} \,X_a)$ implies the existence of functions 
$k_a^b$ defined by $k_a^b\,dj_b=-d\{j_a,h\}$  which are functionally dependent 
on the $j_a$s. Some of the subtleties involved are discussed in \cite{newns}, 
and related global issues are discussed in \cite{karshon-lerman}.} This implies 
that the variation in the 
Hamiltonian vector field under the group action is tangent to the group orbits, 
\[
  {\phi_g}_*X_h-X_h\equiv 0 \,\,(\text{mod} \,X_a),
\]
and consequently the vector field $X_h$ descends to the reduced vector field 
$\bar{X}:=\rho_*X_h$ on the quotient $\bar{P}$. Furthermore, it
implies that the Poisson bracket of a $G$-invariant function with the
Hamiltonian is still a $G$-invariant function.  Said slightly
differently, the mapping
\[
\{\cdot,h\}:C^{\infty}(P)^G \rightarrow C^{\infty}(P)^G : f\mapsto \{f,h\}
\]
is an {\it outer} Poisson morphism (it is an outer morphism because
$h\notin C^{\infty}(P)^G$.)  This gives (singular) reduced dynamics on
the quotient $\bar{P}$ in the Poisson form
\[
  \dot{f}=\{f,h\} \qquad f\in C^{\infty}(P)^G,
\]
as $C^{\infty}(P)^G$ is identified with the smooth functions on $\bar{P}$.

\begin{example}
  Consider a particle moving in linear gravity. This system has the Hamiltonian 
description 
\[
  h= \frac12 p^2 +q.
\]
Then the vector field $Y=\partial_q$ is a symmetry of the symplectic form, has 
momentum $j=p$, and even though the Hamiltonian is not $Y$-invariant, 
$\pounds_Yh\neq0$, the derivative $\pounds_Ydh=0$. Since an invariant function 
is just a function of the variable $p$, the Poisson bracket of an invariant 
function with the Hamiltonian is an invariant function.  Note that 
$\pounds_Ydh=0$ yields, by the magic formula,
\[
\pounds_Ydh = d(Y\lefthook dh) + Y\lefthook(ddh)=  d(Y\lefthook dh)=0,
\]
which implies that $Y\lefthook dh=c$, $c$ a constant. It follows that the 
time evolution of the momentum $j$ is of the form $j(t)=ct+d$, for some 
constants $c$ and $d$.
\end{example}

\begin{example}
  The classical particle with spin.  One may reduce the spinning charged rigid 
body in a magnetic field with the nonlinear constraint of constant length of 
angular momentum to get Souriau's model of a classical particle with spin (see 
\cite{bates-nester} and \cite{cushman-kempainen-sniatycki}.)
\end{example}

\begin{example}
  Collective Hamiltonians \cite{guillemin-sternberg80}.  If the Hamiltonian is 
a pullback of a function on 
the dual algebra by the momentum map, $h=j^*f$, $f:\mathfrak{g}^*\rightarrow 
\R$, then the Casimirs play the role of the invariant functions.
\end{example}

\begin{example}
  Rotating coordinates. Let $G=\SO(2)$ be the group of rotations about the $x^3$ axis in $\R^3$, and $j$ the momentum for the lifted action on the cotangent bundle $T^*\R^3$. Since $G$ acts isometrically for the standard metric on $\R^3$, one sees the addition of the momentum $j$ to the Hamiltonian in the $G$-moving coordinate system.  This is a specific case of the general phenomenon of the addition of collective terms in Hamiltonians viewed in co-moving coordinates along one-parameter isometry groups.  
  \end{example}

\section{Reconstruction (1)}
A general\footnote{This is a general method because it applies to any
  reduced dynamics, Hamiltonian or otherwise. This is the approach taken, for 
example in \cite{cushman-duistermaat-sniatycki}.} way to reconstruct an
integral curve $c(t)$ from an integral curve $\bar{c}(t)=\rho(c(t))$
in the reduced space is to first pull back the $G$-action to any lift
$b(t)$ of $\bar{c}(t)$.  The reconstruction problem is to find a curve
$g(t)$ in the group $G$ so that
\begin{equation} \label{action}
  c(t) = \phi(g(t),b(t))
\end{equation}
satisfies the dynamical equation $\dot{c}=X_h(c)$. Differentiation of
equation (\ref{action}) with respect to the parameter $t$ yields a
non-autonomous differential equation for the group variable:
\[
  D_1\phi(g,b)\cdot \dot{g} +D_2\phi(g,b)\cdot \dot{b} = X_h. 
\]

However, this approach neglects a key component of the Hamiltonian
structure of the system, namely the time dependence of the momenta.  A
refinement of the reconstruction procedure that is not only adapted to
the Hamiltonian structure, but furthermore reduces to the usual
reconstruction procedure in the case when the Hamiltonian is
$G$-invariant, runs as follows.

Observe that the Poisson brackets $\{j_a,h\}=f_a(j_1,\dots,j_n)$ define a 
vector field on the dual of the Lie algebra $\mathfrak{g}^*$ given by the 
differential equations
\[
  \frac{dj_a}{dt}=f_a(j_1,\dots,j_n).
\]
This differential equation is the {\it first reconstruction equation}.  
Denote by $\mu(t)$ an integral curve of this vector field. To find the integral 
curve 
$c(t)$ with initial condition $p=c(0)$ then requires two curves: the first is 
the integral curve $\bar{c}(t)$ of the reduced dynamics with initial condition 
$\rho(p)=\bar{c}(0)$, and the second is the integral curve $\mu(t)$ with initial 
condition $j(p)=\mu(0)$. The curve $\bar{c}(t)$ is then lifted to a curve $b(t)$ 
that satisfies the two constraints $\rho(b(t))=\bar{c}(t)$ and $j(b(t))=\mu(t)$. 
Once again, the dynamical reconstruction problem is to find a curve $g(t)$ in 
the group $G$ so that $c(t) = \phi(g(t),b(t))$ satisfies the dynamical equation 
$\dot{c}=X_h(c)$, but now the curve $g(t)$ lies in the stability group 
$G_{\mu(t)}$.\footnote{In the special case of Hamiltonian reduction, the 
functions $f_a\equiv 0$ because the momenta are constants of motion, so the 
curve $\mu(t)$ is a constant, and hence the stability group $G_{\mu}$ does not 
depend on $t$.} Differentiation of this condition with respect to $t$ yields the 
{\it second reconstruction equation}:
\[
  D_1\phi(g,b)\cdot \dot{g} +D_2\phi(g,b)\cdot \dot{b} = X_h. 
\]
This version of the equation is chosen rather than
(\ref{action}) because the group is smaller, even though it is varying
in time.\footnote{An analogous construction may be found in
  \cite{shnider-sternberg}, in the case of commuting invariance
  groups.}

A special case of the preceding is in some sense quite typical.
Suppose that the stability group $\mu(0)$ is abelian and of minimal dimension, 
and that this is stable in the sense that the stability group
$G_{\mu(t)}$ is also abelian and connected, and so of the form $\R^r\times\T^s$ 
for some
integers $r$ and some $s$. This allows us to choose an interval
$I:=(-\epsilon,\epsilon)$ about $t=0$, an identification of
$G_{\mu(t)}$ with $G_{\mu(0)}$ and thus realize a trivialization of
$\phi(G_{\mu(t)}, b(t))$ as $I\times G_{\mu(0)}$ where the lift $b(t)$
is the product $I\times e$, where $e$ is the identity in the group.  To see how
this might work in practice, consider the following example.

\begin{example}
  (The elliptic particle.)  Consider the Hamiltonian system
  $(P,\omega,h)$ with configuration space $Q=\R^2$, phase space
  $P=T^*Q=T^*\R^2$, projection $\pi:P\rightarrow Q$, symplectic form
  $\omega=dx\wedge dp_x +dy\wedge dp_y$, and Hamiltonian
\[
  h = \frac{1}{2}((1+k^2/2+y^2)p_x^2 -2xyp_xp_y +(1-k^2/2+x^2)p_y^2),
\]
with $0<k<1$.  The Euclidian group $G=\SE(2,\R)$ acts on the configuration space 
$Q=\R^2$, as the group of matrices with affine action
\[
  \phi:G\times Q\rightarrow Q:\left(\begin{pmatrix}
    \cos\theta & -\sin\theta & u \\
    \sin\theta & \cos\theta & v \\
    0 & 0 & 1  
  \end{pmatrix},
   \begin{pmatrix} 
   x \\ y 
    \end{pmatrix}
  \right) \rightarrow 
  \begin{pmatrix}
    \cos\theta \,x -\sin\theta \,y +u \\
      \sin\theta \, x - \cos\theta \, y +v
  \end{pmatrix}.
\]
Define a basis for the Lie algebra $\mathfrak{g}=\se(2,\R)$ by setting
\[
  e_1 = \begin{pmatrix}
          0 & 0 & 1 \\
          0 & 0 & 0 \\
          0 & 0 & 0
        \end{pmatrix}, \quad
     e_2 = \begin{pmatrix}
          0 & 0 & 0 \\
          0 & 0 & 1 \\
          0 & 0 & 0
        \end{pmatrix}, \quad   
     e_3 = \begin{pmatrix}
          0 & -1 & 0 \\
          1 & 0 & 0 \\
          0 & 0 & 0
        \end{pmatrix}.
\]
The matrices $f^1=2e_1^t$, $f^2=2e_2^t$, and $f^3=e_3^t$ form a dual basis 
of 
the dual of the Lie algebra $\mathfrak{g}^*=\se(2,\R)^*$ in the sense that if 
the natural pairing $\mathfrak{g}^*\times\mathfrak{g}\rightarrow \R$ is 
given by the matrix multiplication $\langle\mu,X\rangle=\frac{1}{2}\tr(\mu 
X)$, the pairings
\[
  \langle f^k,e_l\rangle = \delta^k_l.
\]

The action lifts to a Hamiltonian action on phase space with momentum map 
$j:P\rightarrow \mathfrak{g}^*$ with components $j=(j_1,j_2,j_3)$, where $j_1 
=p_x$, $j_2=p_y$, $j_3= yp_x-xp_y$.  In matrices,
\[
  j:P\rightarrow \mathfrak{g}^*:(x,y,p_x,p_y)\rightarrow 
  \begin{pmatrix}
  0 & j_3 & 0 \\
  -j_3 & 0 & 0 \\
  2j_1 & 2j_2 & 0
  \end{pmatrix}.
\]
The components of the momentum map satisfy the Poisson bracket relations
\[
  \{j_1,j_2\}=0, \qquad \{j_2,j_3\}=-j_1, \qquad \{j_3,j_1\}=-j_2.
\]
The Hamiltonian is not invariant under the $G$-action on the phase space, as 
\[
  \{j_1,h\}= j_2j_3, \qquad \{j_2,h\}= -j_1j_3, \qquad \{j_3,h\}= -k^2j_1j_2.
\]
The $G$-invariant functions on $P$ are all functions of 
$\sigma=|p|^2=p_x^2+p_y^2$.  Thus the equation on the reduced space is given 
by the Poisson bracket 
\[
  \dot{\sigma} = \{\sigma,h\}=0,
  \]
which immediately integrates to $\sigma=$ constant.\footnote{The
  reader will immediately observe that the relation $\{\sigma,h\}=0$
  implies the system is completely integrable in the sense of
  Liouville.  However, the construction of action-angle variables is
  more involved than our subsequent analysis.}  The first
reconstruction equation is the differential equation in the dual
algebra
\[
  \frac{dj_1}{dt}=j_2j_3, \qquad 
\frac{dj_2}{dt}=-j_3j_1, \qquad \frac{dj_3}{dt}=-k^2j_2j_3.
\]
For simplicity, the reconstruction will be given for the integral curve with 
initial condition $(x_0,y_0,{p_x}_0,{p_y}_0)=(-1,0,0,1)$, so the initial values 
\[
  (\sigma_0,{j_1}_0,{j_2}_0,{j_3}_0)=(1,0,1,1).
  \]
Denote by $\mu(t)=(\mu_1(t),\mu_2(t),\mu_3(t))$ the solution of this
initial value problem in $\mathfrak{g}^*$.  This implies that
\[
  \sigma(t)=1, \quad \mu_1(t)=\sn(t;k), \quad \mu_2(t)=\cn(t;k), \quad 
\mu_3(t)=\dn(t;k),
\]
where $\sn(t;k)$, $\cn(t;k)$ and $\dn(t;k)$ are the Jacobi elliptic functions.

Now we should examine the second reconstruction equation and the
isotropy subgroup $G_{\mu(t)}$, (the subgroup of $G$ that fixes
$\mu(t)$ under the coadjoint action), which is the one-parameter
subgroup
\[
  G_{\mu(t)} = \{ \exp sX\,|\,s\in\R\}
\]
where
\[
  X=2\mu_1(t)e_1-2\mu_2(t)e_2+\mu_3(t)e_3 \in \mathfrak{g}.
\]
However, we gain a somewhat different insight if we proceed a little differently than a direct application of the theory suggests.  The component $j_3$ of the momentum map implies
\[
\mu_3(t) = y\mu_1(t) - x\mu_2(t),
\]
which is the equation of a moving line in the $xy$-plane. Picking the point $q_0=\mu_3(-\mu_2,\mu_1)$ to be the point on the line nearest the origin, parametrize the moving line as
\[
q(s) = q_0 +s(\mu_1,\mu_2),
\]
where $s$ is an arc length parameter on the line.  A differential equation for the parameter $s$ yielding the reconstruction of the desired integral curve is 
\[
\frac{d}{dt}  [ q(s(t),t)] = \pi_*X_h.
\]
Taking the inner product of this with the unit vector $(\mu_1,\mu_2)$ gives
\[
\begin{split}
\dot{s} = \mu_1\dot{x}+\mu_2\dot{y}  = & \,\,(1+k^2/2+y^2)p_x^2-xyp_xp_y  + \\
                                               &+(1-k^2/2+x^2)p_y^2 -xyp_xp_y\\
                                              = & \,\,1+\mu_3^2+\frac{k^2}{2}(\mu_1^2 -\mu_2^2) \\
                                              = & \,\,2-k^2/2, \quad {\text{a constant!}}
\end{split} 
\]
The reconstructed integral curve $c(t)$ immediately follows.
\end{example}

\section{Reconstruction (2)}
That the previous example had such a pretty solution suggests that
something deeper is at work.\footnote{The reader might suspect that
  this is due to the Hamiltonian being collective.  While correct, our
  view is that this is not the best answer, because it places misleading attention on the invariant functions being Casimirs. See
  \cite{guillemin-sternberg80} for more details.}  Our preferred view is to see the Hamiltonian as a sum of two commuting Hamiltonians,
symbolically written as $h=h_{\sigma}+h_j$, $\{h_{\sigma},h_j\}=0$, thinking of $h_{\sigma}$ as the invariant part and $h_j$ as the collective part.
This implies that the flow of the Hamiltonian $h$ may be found as the
composition of the flows of $h_{\sigma}$ and $h_j$.  Hence, an alternative reconstruction procedure presents itself: reconstruct the dynamics of the invariant part $h_{\sigma}$ using the {\it fixed} subgroup $G_{\mu(t)}$ as well as that of $h_j$ and then compose.  Note, however, that there is no free lunch here.  The flow of $h_{\sigma}$ must be reconstructed treating each point along the flow of $h_j$ as a new initial condition.  In other words, instead of integrating a time-dependent differential equation in which $\mu(t)$ varies, the integration is over a one-parameter family of equations, each of which have constant $\mu$.   

\section{Continuations}

Our discussion has left many avenues unexplored.  Possible further 
explorations include the following 
\begin{enumerate}
\item Assumptions on the group action.  For the sake of brevity, only proper 
group actions were considered, as the theory is now well established.  However, 
in some cases of interest, such as the coadjoint action of a Lie group on the 
dual of its Lie algebra, the action need not be proper.  There are more general 
types of group actions, such as {\it polite} actions (see 
\cite{bates-sniatycki14}) that allow reduction by symmetry and reconstruction in 
terms of differential equations on manifolds.  It would be very interesting to 
see to what extent the theory presented here extends to more general group 
actions.
\item Functional dependence.  This paper avoided all issues of 
functional dependence by assuming that the Poisson bracket $\{j_a,h\}$ was a 
globally defined function of the momenta. It would be interesting to be able 
to weaken this to the condition that the Hamiltonian vector field $X_h$ is 
tangent to the group orbits, as that suffices for the existence of reduced 
dynamics.  The global issues involved are somewhat subtle. For an 
example see theorem $4$ of \cite{karshon-lerman}.
\item Complete integrability.  In the theory of completely integrable 
Hamiltonian systems, the flow is seen to be linearized on tori.  This means that 
there is a local action of a torus group under which the flow is invariant.  The 
generalization of this theory to not-quite-Hamil\-tonian systems is unclear, as 
there is presently no precise notion of what a completely integrable 
not-quite-Hamil\-tonian system should be.
\item Geometric quantization.  It is of interest to understand to what extent 
quantization and reduction commute in the case under consideration.
\end{enumerate}

\vspace{20pt}

\noindent We wish to give our thanks to the referee whose careful 
dissection of an earlier version of this manuscript has resulted in a much 
improved version.

\bibliographystyle{plain}

\vspace{20pt}

\vspace{20pt}

\noindent Larry M. Bates \newline
Department of Mathematics \newline
University of Calgary \newline
Calgary, Alberta \newline
Canada T2N 1N4 \newline
bates@ucalgary.ca

\vspace{20pt}

\noindent J\k{e}drzej \'Sniatycki \newline
Department of Mathematics \newline
University of Calgary \newline
Calgary, Alberta \newline
Canada T2N 1N4 \newline
sniatyck@ucalgary.ca
\end{document}